\documentclass[11pt]{siamltex}

\usepackage{datetime}
\settimeformat{ampmtime}

\newcommand{\nwc}{\newcommand}
\nwc{\draftdate}{\today}
\usepackage[utf8]{inputenc}

\usepackage[pdftex]{graphicx}
\usepackage{psfrag}
\usepackage{wrapfig}
\usepackage{epsf}
\usepackage{amsmath,amssymb}
\usepackage{color}

\usepackage{enumitem}
\usepackage{hyperref}
\usepackage[font={footnotesize}]{caption} 

\allowdisplaybreaks

\renewcommand{\theenumi}{{\rm (\roman{enumi}) }}

\newtheorem{remark}[theorem]{Remark}
\newcommand{\disp}{\displaystyle}

\newcommand{\barint}{\hbox{$\int$\kern-0.75\intwidth
		\vrule width 0.5\intwidth height 2.4pt depth -2pt\kern0.25\intwidth}}
\newlength\intwidth
\setbox0=\hbox{$\int$}
\intwidth=\wd0

\newcommand\avint{\hbox{\hbox{$\displaystyle \int$}\hbox{\kern-.9em{$-$}}}}

\newcommand\smavint{\hbox{\hbox{$\int$}\hbox{\kern-.75em{$-$}}}}

\nwc{\st}{^{\mbox{\it st}}}

\nwc{\qref}[1]{(\ref{#1})}
\nwc{\veloc}{v}
\nwc{\rhoc}{\beta}
\nwc{\hl}{\hat{L}}
\newcommand{\na}{\nabla}

\def\Xint#1{\mathchoice
	{\XXint\displaystyle\textstyle{#1}}%
	{\XXint\textstyle\scriptstyle{#1}}%
	{\XXint\scriptstyle\scriptscriptstyle{#1}}%
	{\XXint\scriptscriptstyle\scriptscriptstyle{#1}}%
	\!\int}
\def\XXint#1#2#3{{\setbox0=\hbox{$#1{#2#3}{\int}$}
		\vcenter{\hbox{$#2#3$}}\kern-.51\wd0}}

\def\dashint{\Xint-}

\nwc{\intRp}{\int_0^\infty}
\nwc{\aint}{\dashint}
\nwc{\aaint}{\dashint}

\newcommand{\De}{\Delta}

\newcommand{\vep}{\varepsilon}

\newcommand{\Om}{\Omega}

\nwc{\kn}{{\kappa_n}}
\nwc{\kg}{{\kappa_g}}
\nwc{\tg}{{\tau_g}}
\nwc{\vkappa}{{\vec{\kappa}}}

\newcommand{\bn}{{\bf n}}

\newcommand{\R}{\mathbb R}
\newcommand{\be}{\begin{eqnarray}}
	\newcommand{\ee}{\end{eqnarray}}
\newcommand{\nn}{\nonumber}
\newcommand{\ben}{\begin{eqnarray*}}
	\newcommand{\een}{\end{eqnarray*}}

\newcommand{\stophere}{\newpage
	\bibliographystyle{plain}
	\bibliography{dai-bib,dai-bib-nanoparticles}
\end{document}}

\title{Gamma Convergence for the \\{de Gennes}-Cahn-Hilliard energy
\thanks {{\bf Funding:} The work of the first author was partially supported by the U.S. National Science Foundation through grant
DMS-1815746. The work of the third author was partially supported by the U.S. National Science Foundation through grant
DMS-2012634.}}
\author{
Shibin Dai
\thanks{Department of Mathematics, The University of Alabama, Tuscaloosa, AL 35487-0350, USA. 
	Email: sdai4@ua.edu}
\and 
Joseph Renzi
\thanks{Department of Mathematics, The University of Alabama, Tuscaloosa, AL 35487-0350, USA. 
	Email: jdrenzi@crimson.ua.edu} 
\and 
Steven M. Wise
\thanks{Department of Mathematics, 
	The University of Tennessee, Knoxville, 
	227 Ayres Hall. 1403 Circle Drive. Knoxville TN 37996-1320, USA. 
	Email: swise1@utk.edu} 
}

\date{}

\begin{document}

\maketitle

\begin{abstract}
	The degenerate de Gennes-Cahn-Hilliard (dGCH) equation is a model for phase separation which may more closely
	approximate surface diffusion than others  in the limit when the thickness of the transition layer approaches zero. As 
	a first step to understand the limiting behavior, in this paper we study the $\Gamma$--limit of the dGCH energy. We 
	find that its $\Gamma$--limit is a constant multiple of the interface area, where the constant is determined by the 
	de Gennes coefficient together with the double well potential. In contrast, the transition layer profile is solely determined 
	by the double well potential. 
	
\end{abstract}

\section{Introduction} 

Surface diffusion is a fourth-order, highly nonlinear geometric evolution of a surface in which the normal velocity is proportional to 
the surface Laplacian of the mean curvature. It plays an important role in materials science, a common example being solid 	
state dewetting, see, for 
example, \cite{thompson12, thompson-dewetting}. 
Direct numerical approaches to surface diffusion have been developed. See, e.g.,  
\cite{bjwz-2017, barrett2019, bmn2005,hausser-voigt-2005, hausser-voigt-2007}. But, they inevitably face the difficulties 
of handling topological changes when merging or splitting occurs. In comparison, 
diffuse interface approaches capture the motion of interfaces implicitly and 
handle topological changes without any difficulty. Based on the intuition that the degenerate Cahn-Hilliard equation 
converges to  surface diffusion when the transition layer thickness approaches 
zero~\cite{cahn-taylor:surface, taylor-cahn:jsp1994}, diffuse-interface approximations for surface diffusion are a reasonable 
alternative~\cite{banas09,BSB+16,JBTS12, LLRV09, NBS+17, RRV06, SBB+15, SKSV17, Low-09, WKL07, WLK+05, YCG06}.

The conventional Cahn-Hilliard (CH) functional 
\vspace{-.1in}
\begin{align}
	E^\vep_{\rm CH}(u) = \int_\Om \left\{\frac{\vep}{2}|\nabla u|^2 + \frac1\vep W(u)\right\}dx \label{def-E}
\end{align}
is  a widely used phenomenological diffuse-interface model to describe the free energy of a system 
that goes through phase separation \cite{CH-58}, \cite{CH-reprise}. Here $\Omega\subset\R^n$ is a bounded domain 
in $\R^n$, with $n\geq 2$, $u:\Omega\to \R$ is the relative concentration 
of the two phases, $W(u)$ is a double-well potential with two equal minima at $u^-< u^+ $ corresponding to 
the two pure phases, 
and $\vep>0$ is a parameter that is proportional to the thickness of the transition region between the two phases. 
For convenience and for systems where $u$ is only used as an order parameter to indicate the phases,  
$W$ is usually taken to be smooth, such as
\begin{align} \label{def-W}
	W(u) = \gamma (u-u^+)^2(u-u^-)^2.
\end{align}
Here $\gamma$ is a constant which usually takes a value such that $\int_{u^-}^{u^+} \sqrt{2W(s)}\,ds =1$. For simplicity 
we take $\gamma=1$.
The Cahn-Hilliard equation 
\begin{align}
	\partial_t u &= \nabla \cdot \left( M(u) \nabla \mu\right)  
	\quad\mbox{ for }x\in\Om\subset\R^n,\;t\in [0,\infty), \label{CH1}  \\
	\mu  &:= -\vep^2\De u +  W'(u), \label{CH2}
\end{align}
coupled with either Neumann boundary conditions, 
$\partial_\bn u= \partial_\bn \mu =0$, or 
periodic boundary conditions, where $\Om$ is the periodic cell, formally dissipates the 
free energy \qref{def-E}. 
The diffusion mobility $M(u)$ is 
nonnegative and generally depends on $u$. 
The modeling of the dependence of the diffusion mobility $M(u)$ on the phase concentration $u$ presents 
additional challenges,
especially when $M(u)$ is degenerate in the two pure phases $u^\pm$. For instance, $M(u)$ may take the form 
\cite{cahn-taylor:surface, puri97, taylor-cahn:jsp1994}
\begin{align} \label{def-M}
	M(u)=|(u-u^+)(u-u^-)|^m \quad\mbox{ for all } u\in \R
\end{align}	
for some $m>0$. Due to the degeneracy, 
it has been conjectured that  there is no diffusion in the two pure phases,
and the dynamics is governed by diffusion only in the transition region \cite{cahn-taylor:surface}, \cite{taylor-cahn:jsp1994}.
In addition,  if the initial value of $u$ lies inside $[u^-, u^+]$, then the solution $u$ is believed to remain  
in $[u^-, u^+]$, for all time.   
However, recent studies by one of the authors (Shibin Dai) and his collaborator Qiang Du 
\cite{dd:onesidedCH,dd:degenerateCH,dd:numericCH,dd:weakCH}, 
and by Lee, M\"{u}nch, and S\"{u}li
\cite{lms:insufficient,lms:degenerateCH}, show that in the 
sharp interface limit of the degenerate Cahn-Hilliard equation,  there is an undesired diffusion process in the bulk phase, in addition
to the desired surface diffusion along the interface. Dai and Du  studied the coarsening mechanisms by asymptotic
analysis \cite{dd:onesidedCH, dd:degenerateCH}, the coarsening rates by careful and detailed 
numerical simulations \cite{dd:numericCH}, and the wellposedness by theoretical analysis 
\cite{dd:weakCH}. Their studies show that for 
a smooth potential such as \qref{def-W} and a degenerate mobility \qref{def-M}, the sharp interface limit of the degenerate CH
equations is not surface diffusion. Indeed, there is a porous medium diffusion process in both phases, in 
addition to surface diffusion. More precisely, for the choice $M(u)=|(u-u^+)(u-u^-)|$, the motion of the interfaces occurs in the 
$t_2=O(\vep^{-2})$ time scale and it can be written as 
\begin{align}
	\nabla\cdot & (|\mu_1|\nabla\mu_1) =0 \quad\mbox{in }\Omega_\pm, \label{degCH-eq1}\\
	\mu_1 &= -H \quad\mbox{on }\Gamma, \label{degCH-eq2} \\  
	V_\bn &   =  \Delta_s H - \biggl[|\mu_1|\partial_\bn\mu_1\biggr]^+_-
	\quad\mbox{on }\Gamma. \label{degCH-eq3}
\end{align}
Here $\Omega_+$ and $\Omega_-$ are the regions occupied by the positive and negative phases, 
respectively, $\Gamma$ is the interface separating the two phases, $H$ is the mean curvature, and 
$\bn$ is the normal vector, pointing toward $\Omega_+$. If $\Omega_+$ consists of convex components, 
then $H<0$. Due to \qref{degCH-eq1} and \qref{degCH-eq2}, we have
$\mu_1>0$ in $\Omega_\pm$. Since $W''(u^\pm)>0$ and  $u=u^\pm + \vep \mu_1/W''(u^\pm)  + O(\vep^2)$, 
we see that $u>u^+$ in $\Omega_+$ and  $u>u^-$ in $\Omega_-$.  
Consequently, the solution $u$ is not confined in $[u^-, u^+]$
and the normal velocity of the interface is determined by surface diffusion together with a quasi-stationary 
porous medium diffusion process in both phases. If the degeneracy in the mobility is higher, then the porous 
medium diffusion process can be weakened into a process in a slower time scale, but it will not be annihilated. 
See also \cite{lms:insufficient, lms:degenerateCH,lms:response,voigt:comment} for related results and 
discussions. 

Recently one of the authors (Wise) and his collaborators Salvalaglio and Voigt introduced a new variational diffuse
interface model \cite{Salvalaglio2020a}, 
in which they imposed a singularity in the formulation of the 
energy functional. 
Rather than formulating the free energy as in \qref{def-E}, they define the free energy 
by including a de Gennes coefficient \cite{DWZZ19, LQZ17}. In this paper we call it the de Gennes-Cahn-Hilliard (dGCH) energy
\begin{align}
	E^\vep_{\rm dGCH}(u) := \int_\Omega \frac1{g_0(u)}\left( \frac\vep2|\nabla u|^2 + \frac1\vep W(u)\right)dx, \quad\mbox{ for all }u\in H^1(\Omega).
\end{align}
Here the double well potential $W(u)$ is defined by \qref{def-W}, and $g_0$ is a function of the form 
\begin{align}\label{def_g0p}
	g_0(u) =  |(u-u^-)(u-u^+)|^p, \quad p>0.
\end{align}
Of particular interest is the case when $p=1$,
\begin{align}\label{def_g01}
	g_0(u) =  |(u-u^-)(u-u^+)|.
\end{align}
The factor $\frac1{g_0}$ is called the   de Gennes coefficient \cite{DWZZ19, LQZ17}, or the energy restriction function. 
The corresponding energy dissipation flow is the degenerate de Gennes-Cahn-Hilliard system (which was called 
the doubly degenerate Cahn-Hilliard system in \cite{Salvalaglio2020a}).
\begin{align}
	\partial_t u &= \frac1\vep\nabla\cdot(M_0(u)\nabla \mu), \label{dGCH-eq1}\\
	\mu &=  - \vep \nabla\cdot\left(\frac{1}{g_0(u)} \nabla u\right) + \frac1{\vep g_0(u)} W'(u)
	-\frac{g_0'(u)}{g_0^2(u)}\left(\frac\vep2 |\nabla u|^2 + \frac1\vep W(u)\right), \label{dGCH-eq2} 
\end{align}
where $\mu= \delta_u E^\vep_{\rm dGCH}$ is the chemical potential, defined by 
the variational derivative of  $E^\vep_{\rm dGCH}$ with respect to $u$. The mobility
$M_0(u)$ is degenerate at $u^\pm$ and is defined as
\begin{align}\label{def-M0}
	M_0(u) =  (u-u^-)^2(u-u^+)^2.
\end{align}  
The factor $g_0'/g_0^2$ is also singular. 
For $u^-<u<u^+$ we have
\begin{align}
	\frac{g_0'(u)}{g_0^2(u)} =- \frac{p(2u-u^+-u^-)}{ (u-u^-)^{p+1} (u^+-u)^{p+1}}.
\end{align} 
Intuitively, the singularities at $u^\pm$ help to keep solutions confined in $[u^-, u^+]$. But the validity of this argument
remains open.  
The degeneracy and singularities in $M_0$, $1/{g_0}$, and ${g_0'}/{g_0^2}$ 
make numerical simulations very challenging, even for $p=1$. One way to ease such challenges is to consider a regularization corresponding to
\begin{align}
	M_\alpha(u) &=  (u-u^-)^2 (u-u^+)^2 + \alpha\vep, \label{def-M_alpha} \\
	g_\alpha(u) &= {\sqrt{(u-u^-)^{2}(u-u^+)^{2} + \alpha^2\vep^2}},  \quad\alpha>0. 
	\label{def-g_alpha}
\end{align}
The regularized system is non-degenerate and formally recovers the dGCH system \qref{dGCH-eq1} and \qref{dGCH-eq2} 
when $\alpha\to 0$. However, the most significant 
difference is that the regularization removes the mechanism which hopefully confines the values of 
$u$ to be in $[u^-,u^+]$. 

Numerical simulations and formal asymptotic analysis in 
\cite{Salvalaglio2020a} show that the 
degenerate dGCH equation \qref{dGCH-eq1}--\qref{dGCH-eq2} 
may more closely approximate surface diffusion than others. This exciting discovery motivated us to 
systematically explore properties of the dGCH model. In this paper we study the sharp interface limit of the dGCH energy as $\vep\to 0$
in the framework of $\Gamma$-convergence. To be more precise, we will show that the $\Gamma$-limit of the 
dGCH energy $E_{\rm dGCH}^\vep$ is a multiple of the 
interface separating the two phases $\{u\approx u^+\}$ and $\{u\approx u^-\}$ 
when $g_0(u)$ is as stated in \qref{def_g0p}, \qref{def_g01}, or when it is regularized as $g_\alpha(u)$ 
in \qref{def-g_alpha}. 
When considering  \qref{def-g_alpha} we will be using $p=1$ because it is computationally the most interesting form. 
This is the first step for us to understand the structure of the dGCH equation, and help pave the path to eventually settle the debate 
about the relation between phase-field models and the surface diffusion. To start with, we extend the definition of 
$E_{\rm dGCH}^\vep(u)$ to all $u\in L^1(\Omega)$ by the following generalization. 
\begin{align}\label{def-EdGCH}
	E^\vep_{\rm dGCH}(u) := \left\{
	\begin{array}{ll}
	\int_\Omega \frac1{g(u)}\left( \frac\vep2|\nabla u|^2 + \frac1\vep W(u)\right)dx, & \mbox{if }u\in H^1(\Omega), \\
	+\infty, &\mbox{otherwise}.
	\end{array}
	\right.
\end{align}
Here the double well potential $W(u)$ is taken as \qref{def-W}, and  
$g(u)$ is taken as  either $g_0(u)$ in \qref{def_g0p} or \qref{def_g01}, or the regularized form $g_\alpha(u)$ in \qref{def-g_alpha}.

\begin{theorem} \label{th-main} The $\Gamma$-limit of $E^\vep_{\rm dGCH}$ is a 
multiple of the perimeter of the set $A$ on which 
$u$ takes the value $u^+$. To be precise, 
the  $\Gamma$-limit of $E^\vep_{\rm dGCH}$
under the strong $L^1(\Omega)$ topology 
is
\vspace{-.1in}
\begin{align}
	E^0_{\rm dGCH}(u) := \left\{ 
	\begin{array}{ll}
		\sigma(p) Per(A) & \mbox{if } u = u^- + (u^+-u^-)\chi_A \in BV(\Omega), \\
		\infty &\mbox{otherwise.}
	\end{array}
	\right.
\end{align}
Here $\chi_A$ is the characteristic function of a set $A$ of finite perimeter, $Per(A)$ is the perimeter of $A$, and 
\[\sigma(p)=\sqrt2\int_{u^-}^{u^+}|(s-u^-)(s-u^+)|^{1-p}\,ds.\] 
In particular, the regularization of $g$ in \qref{def-g_alpha} does not change the $\Gamma$-limit. 

\end{theorem}

\begin{remark}
Due to the singularity, $E^\vep_{\rm dGCH}(u)$ is not necessarily finite for all $u\in H^1(\Omega)$. It is an interesting question to explore 
the domain of $E^\vep_{\rm dGCH}(u)$, but that is not the purpose of this paper.
\end{remark}

\begin{remark}
	When studying mass-preserving phase separations, it is natural to include a mass constraint, 
	$\int_\Omega u\,dx = \mbox{\rm const}$. The corresponding $Gamma$ convergence results can be 
	derived almost the same way as what we do here, except that we need to do some $O(\vep)$ adjustment when 
	constructing the recovery sequence. We omit the details. 
\end{remark}

To find the $\Gamma$-limit of a sequence of functionals, we need three results
(see, e.g., \cite{braides:gamma, leoni:note}): 
\begin{enumerate}
\item a compactness result that help us characterizes the limit functional; 
\item a liminf inequality which tells us what the $\Gamma$-limit should be; and 
\item a recovery sequence that satisfies a limsup inequality, which tells us under what scenario the $\Gamma$-limit 
can be attained.  
\end{enumerate}
In Section 2 we will state and prove the necessary compactness result. In Section 3 we will state and prove the liminf and limsup inequalities. Finally in 
Section 4 we will have a brief discussion about further topics to explore. 

\section{Compactness}\label{sec-compactness}
Throughout the rest of this paper, $g(u)$ takes the form of \qref{def_g0p}, \qref{def_g01}, or \qref{def-g_alpha}. We let 
$\vep_k$ be a positive sequence such that $\vep_k \rightarrow 0$. Without loss of generality we assume $\vep_k<1$ for all 
$k$.

\begin{lemma}\label{Compactness_Lemma}
{(Compactness)} If $u_k$ is a sequence of functions such that \\
$\sup_k E^{\vep_k}_{\rm dGCH}(u_k) < \infty$, then there exists a subsequence, 
$u_{k_j}$, and a measurable set $A$ of finite perimeter in $\Omega$ such that $u_{k_j}$ converges  to 
the following piecewise function
\[ u^- + (u^+-u^-)\chi_A= \left\{ 
	\begin{array}{ll}
		u^+ & \mbox{ if }x\in A, \\
		u^- &\mbox{ otherwise}
	\end{array}
\right.\]
strongly in $L^1(\Omega)$ and a.e. in $\Omega$.  

\end{lemma}

\begin{proof} Suppose $\{u_k\}$ is a sequence in $L^1(\Omega)$  
such that $E^{\vep_k}_{\rm dGCH}(u_k)$ is uniformly bounded. 
Then there exists $M>0$ that is
independent of $k$ such that 
\begin{align}\label{bound1}
	E^{\vep_k}_{\rm dGCH}(u_k) := \int_\Omega \frac1{g(u_k)}\left( \frac{\vep_k}2|\nabla u_k|^2 + \frac1{\vep_k} W(u_k)\right)dx
	\leq M<\infty.
\end{align}
Hence 
\begin{align}\label{bound2}
	\disp \int_{\Omega} \frac{W(u_k)}{g(u_k)} \leq M\vep_k
\end{align} for all $k$, regardless of whatever form $g$ takes. The case $p=1$ is relatively straightforward, and has clean
structures for the singular form \qref{def_g01} and the regularized form \qref{def-g_alpha}. We will first analyze this case and then 
study the general case $0<p<3/2$. 

{\it Case 1. $p=1$ and $g$ is in the singular form of \qref{def_g01}.  } 

For $g_0$ in \qref{def_g01}, by the definition of $W$ in \qref{def-W}, Inequality \qref{bound2}
becomes
\begin{align} \label{bound2.1}
	\disp \int_{\Omega} |(u_k-u^-)(u_k-u^+)|\,dx = \int_\Omega \frac{W(u_k)}{g_0(u_k)}\,dx \leq  M\vep_k  \quad\mbox{for all } k.
\end{align}
Since $\Omega$ is bounded set, 
it is straightforward to conclude that $\disp\int_{\Omega} |u_k|^2\,dx\leq C<\infty$ 
for some constant $C>0$ independent of $k$. Hence  $\{u_k\}$ is bounded in $L^2(\Omega)$ and consequently bounded in 
$L^1(\Omega)$: there exists 
$C>0$ that depends on $\Omega$ but
independent of $k$ such that  
\begin{align}\label{bound3}
	\disp\int_\Omega |u_k|\,dx\leq C<\infty \mbox{ for all }k.
\end{align}

Now we estimate $\{\na u_k\}$ in $L^1(\Omega)$.  By the AM-GM inequality, \qref{bound1} implies that 
\begin{align}\label{bound3.1}
	E^{\vep_k}_{\rm dGCH}(u_k)
	&\geq \int_\Omega \frac{2}{g_0(u_k)} \left( \frac{\vep_k}2|\na u_k|^2 \cdot \frac1{\vep_k} W(u_k) \right)^{1/2} 
	= \sqrt{2}\int_\Omega  |\na u_k|\,dx.
\end{align}
Thus
\begin{align} \label{bound4}
	\int_\Omega |\nabla u_k|\,dx \leq M/\sqrt2.
\end{align}
Combining \qref{bound3}  and \qref{bound4}, 
we conclude that $u_k$ is bounded in $W^{1,1}(\Omega)$. Hence there exists a subsequence $u_{k_j}$ and a function $u\in BV(\Omega)$ such that 
\begin{align}
	&u_{k_j} \to u \quad\mbox{ strongly in } L^1(\Omega) \mbox{ and a.e. in }\Omega,  \\
	&\int_\Omega |\nabla u| \,dx \leq \liminf_{j\to\infty} \int_\Omega |\nabla u_{k_j}|\,dx. \label{bound4.2}
\end{align} 
Here $\int_\Omega |\nabla u|\,dx$ is the BV-seminorm of $u$. By  \qref{bound2.1} and Fatou's lemma, we have 
\begin{align}
	\int_\Omega |(u-u^-)(u-u^+)|\,dx \leq \liminf_{j\to\infty}\int_\Omega |(u_{k_j} - u^-)(u_{k_j} - u^+)| =0.
\end{align}
So $u$ only takes two values $u^\pm$. We can write 
\begin{align}
	u = u^- + (u^+-u^-)\chi_A= \left\{ 
	\begin{array}{ll}
		u^+ & \mbox{ if }x\in A, \\
		u^- &\mbox{ otherwise}
	\end{array}
\right.
\end{align}
where $A$ is a set of finite perimeter. 


{\it Case 2. $p=1$ and $g$ is in the regularized form of \qref{def-g_alpha}.  } 

For the regularized form $g_\alpha$ in \qref{def-g_alpha}, define 
\begin{align}\label{Omega-reg}
	\Omega^{r}_{k}:=\{x \in \Omega: [(u_{k}-u^-)(u_{k}-u^+)]^2\leq \alpha^2 \vep_k^2 \}
\end{align}
Then 
\begin{align}\label{bound5}
	\int_{\Omega_k^r} |u-u^-| |u-u^+|\,dx \leq |\Omega_k^r| \alpha\vep_k \leq \alpha |\Omega|\vep_k.
\end{align}	
In $\Omega\setminus \Omega^r_k$, we have 
\[ [(u_k-u^-)(|u_k-u^+)]^2 \geq \frac12[(u_k-u^-)(u_k-u^+)]^2 + \frac12\alpha^2\vep^2  = \frac12 g_\alpha^2 (u_k).\]
Hence
\begin{align}
	&\int_{\Omega\setminus\Omega_k^r} |u_k-u^-||u_k-u^+|\,dx 
	= \int_{\Omega\setminus\Omega_k^r} \frac{W(u_k)}{|u-u^-| |u-u^+|}\,dx \nn\\
	&\leq \int_{\Omega\setminus\Omega_k^r}\frac{\sqrt2 W(u_k)}{g_\alpha(u_k)}
	\leq \int_{\Omega}\frac{\sqrt2 W(u_k)}{g_\alpha(u_k)}
	\leq \sqrt2 M\vep_k. \label{bound6}
\end{align}
Combining \qref{bound5} and \qref{bound6} we have 
\begin{align} \label{bound6.1}
 	\int_\Omega |u_k-u^-| |u_k-u^+|\,dx \leq (\alpha|\Omega| + \sqrt2M)\vep_k \quad\mbox{ for all }k.
\end{align} 
Hence $\{u_k\}$ is bounded in $L^2(\Omega)$ and consequently bounded in $L^1(\Omega)$.

To control $\nabla u_k$, we notice that over $\Omega^{r}_k$, 
\begin{align*}
	M\geq E^{\vep_k}_{\rm dGCH}[u_{k}] &\ge \displaystyle\int_{\Omega^r_k} 
	\frac{\vep_k | \na u_{k} |^2}{2\sqrt{[(u_{k}-u^-)(u_{k}-u^+)]^2+\alpha^2 \vep_k^2}}
	\ge \int_{\Omega^r_k} \frac{|\nabla u_{k}|^2}{2\sqrt{2}\alpha}.
\end{align*}
Thus $\int_{\Omega_k^r} |\na u_{k}|^2\,dx \leq 2\sqrt2 \alpha M$ and 
\begin{align} \label{bound7}
	\int_{\Omega_k^r} |\na u_{k}|\,dx \leq  (2\sqrt2 \alpha M |\Omega_k^r|)^{1/2} \leq (2\sqrt2 \alpha M |\Omega|)^{1/2}.
\end{align}	
For  $\Omega\setminus \Omega^{r}_k$, we have
\begin{align}
	M&\geq E^{\vep_k}_{\rm dGCH}[u_{k}] \nn\\
	&\geq \displaystyle\int_{\Omega\setminus\Omega^r_k} \left\{\frac{\vep_k | \na u_{k} |^2}{2\sqrt{[(u_{k}-u^-)(u_{k}-u^+)]^2
	+\alpha^2 \vep_k^2}} \right.\nn\\
	&\quad+\left.\frac{[|(u_{k}-u^-)(u_{k}-u^+)|]^2}{\vep_k \sqrt{[(u_{k}-u^-)(u_{k}-u^+)]^2+\alpha^2 \vep_k^2}}\right\}\,dx \nn\\
	&\geq \int_{\Omega\setminus\Omega^{r}_k} \left\{\frac{\vep_k|\nabla u_{k}|^2}{2\sqrt{2[(u_{k}-u^-)(u_{k}-u^+)]^2}}
	+\frac{|(u_{k}-u^-)(u_{k}-u^+)|}{\vep_k \sqrt{2}}\right\}\,dx \nn\\
	&\geq \int_{\Omega\setminus\Omega_k^r} |\nabla u_k|\,dx \quad\mbox{by the AM-GM inequality.} \label{bound8}
\end{align}
Combining \qref{bound7} and \qref{bound8}, we see that $\{\na u_{k}\}$ is bounded in $L^1(\Omega)$. Combined with the $L^1(\Omega)$ bound of $u$, 
we conclude that $u$ is bounded in $W^{1,1}(\Omega)$. The rest is similar to Case 1 except that we will use \qref{bound6.1} 
in place of \qref{bound2.1}.

{\it Case 3. $(0<p\leq3/2)$ and $g$ is in the singular form \qref{def_g0p}.}

For the general case of $g$ defined by \qref{def_g0p}, inequality 
\qref{bound2} translates into 
\begin{align}\label{bound9}
 	\int_\Omega |u_k-u^+|^{2-p} |u_k-u^-|^{2-p} \,dx\leq M\vep_k \quad\mbox{ for all }k.
\end{align}
First, we show that $\{u_{k}\}$ are equi-integrable.  Let $\gamma > 0$ be fixed.  Let $k_0$ be large enough so that 
\begin{align}\label{bound10}
 	M \vep_k <\lambda^{4-2p} \quad\mbox{ for all } k\geq k_0.
\end{align}	
For any $k\geq k_0$ define 
\[ \Omega_k^1 := \left\{x\in\Omega: |u_k - u^-|\geq |u^-| \mbox{ and } |u_k-u^+|\geq |u^+| \right\}.\]
Then in $\Omega_k^1$ we have
\[ |u_k| \leq |u_k-u^-| + |u^-| \leq 2 |u_k-u^-|\; \mbox{ and }\;
 |u_k| \leq |u_k-u^+| + |u^+| \leq 2 |u_k-u^+|.\]
Consequently in $\Omega_k^1$
\[ |u_k| \leq 2|u_k-u^-|^{1/2} |u_k-u^+|^{1/2}\]
and hence for all $k\geq k_0$
\begin{align}
	\int_{\Omega_k^1} |u_k|\,dx &\leq 2 \int_{\Omega_k^1}|u_k-u^-|^{1/2} |u_k-u^+|^{1/2}\,dx \nn\\
		&\leq  2\int_{\Omega}|u_k-u^-|^{1/2} |u_k-u^+|^{1/2}\,dx \nn\\
		&\leq C \left(\int_{\Omega} |u_k-u^-|^{2-p} |u_k-u^+|^{2-p} \,dx \right)^{1/(4-2p)} \nn\\
		&\leq C\lambda \quad\mbox{ by }\qref{bound9} \mbox{ and }\qref{bound10}. \label{bound11}
\end{align}
Here $C$ depends on $\Omega$ and $p$ but not on $k$. Now 
\[ \Omega\setminus\Omega_k^1 = \{x\in\Omega:  |u_k-u^-|< |u^-|\mbox{ or }|u_k-u^+|< |u^+|\}.\] 
So in $\Omega\setminus\Omega_{k}^1$ we have either 
\[ |u_k|\leq |u_k -u^-| + |u^-| \leq 2 |u^-| \] or \[|u_k| \leq |u_k-u^+| + |u^+| \leq 2 |u^+|.\]
Hence $|u_k| \leq 2 (|u^-| + |u^+|) $ in $\Omega\setminus\Omega_k^1$. For any subset 
$G\subset \Omega\setminus\Omega_k^1$ with $|G| \leq \frac{\lambda}{2(|u^+|+|u^-|)}$, 
\begin{align}\label{bound12}
	\int_G |u|\,dx &\leq 2 (|u^-| + |u^+|) |G| \leq \lambda. 
\end{align}
By \qref{bound11} and \qref{bound12}, there exists constant $C$ depending only on $\Omega$ and $p$ and 
independent of $k$ such that for any subset $G\subset\Omega$
with $|G|\leq \frac{\lambda}{2(|u^+|+|u^-|)}$ we have for all $k\geq k_0$
\begin{align}
	\int_G |u_k| &\leq \int_{G\cap \Omega_k^1} |u_k|\,dx + \int_{G\setminus\Omega_k^1} |u_k|\, dx \leq C\lambda. 
\end{align}
For the first $k_0-1$ terms of $u_k$, 
by the absolute continuity of integrals, there exists $\delta >0$ such that for any subset $G\subset\Omega$, if
$|G|\leq \delta$, then 
\begin{align} 
	\int_G |u_k|\,dx \leq \lambda, \quad k=1,2,\dots, k_0-1. 
\end{align}	
This estimate together with \qref{bound12} shows that the whole sequence $u_k$ is equi-integrable. 

Next we show that there is a subsequence of $u_k$ that converges a.e. in $\Omega$.  
By the AM-GM Inequality, 
\begin{align}\label{bound12.1}
	M\geq E^{\vep_k}_{\rm dGCH}[u_{k}] &\ge \sqrt2\int_{\Omega} |u_{k}-u^-|^{1-p}|u_{k}-u^+|^{1-p}|\na u_{k}|.
\end{align}
 Let $K>0$ be some fixed real number. Define 
 \begin{align}\label{def-h}
 	h_K(t)=\sqrt2\int_{u^-}^t\min\{ |s-u^-|^{1-p}|s-u^+|^{1-p},K\}ds.
 \end{align}
 Then $h_K(t)$ is a strictly increasing and Lipschitz function. Define  
 \[ w_k(x):= h_K(u_k(x)).\]
 Then  $|w_k| \leq \sqrt2 K |u_k|$ and by the chain rule,  
\begin{align} |\nabla w_k| &= \sqrt 2 \min\{ |u_k-u^-|^{1-p}|u_k-u^+|^{1-p},K\}|\nabla u_k| \nn\\
	&\leq \sqrt2|u_k-u^-|^{1-p}|u_k-u^+|^{1-p} |\nabla u_k|. \label{bound12.2}
\end{align}
By \qref{bound12.1} 
 $w_k$ is bounded in $W^{1,1}(\Omega)$.  Hence there exists a 
 subsequence $w_{k_j}$ that converges strongly  in $L^1(\Omega)$ and a.e. in $\Omega$ to a BV function $w(x)$.  
 Note that $h_K$ is strictly increasing and continuous, so $h_K^{-1}$ exists and is continuous, and the corresponding 
 subsequence $u_{k_j} = h^{-1}(w_{k_j})$ converges a.e. in $\Omega$ to  $u:= h_K^{-1}(w)$. 
 By the Egorov theorem and the equi-integrability of $\{u_{k_j}\}$, we see that $u_{k_j}$ strongly converges 
 in $L^1(\Omega)$ to $u$.

Now we show that this limit function $u$ is in fact $BV(\Omega;\{u^-,u^+\})$.  By \qref{bound2}, 
\begin{align*}
	 \int_{\Omega} |u_{k_j}-u^-|^{2-p}|u_{k_j}-u^+|^{2-p}\, dx \leq M \vep_{k_j}.
\end{align*}
By Fatou's lemma, 
\begin{align*}
	 \int_{\Omega} |u-u^-|^{2-p} |u-u^+|^{2-p}\, dx \leq \liminf_{j \rightarrow \infty}  
	 \int_{\Omega} |u_{k_j}-u^-|^{2-p}|u_{k_j}-u^+|^{2-p} = 0.
\end{align*}
So $u$ only takes values of $u^-$ or $u^+$, and $w=h_K(u)$ takes only two values, 0, and $h_K(u^+)$. Since $w\in BV(\Omega)$, 
we can write $w= h_K(u^+)\chi_A$, where $A$ is a set of finite perimeter.  Consequently
$u =  u^- + (u^+-u^-)\chi_A.$
\end{proof}

\section{Liminf and Limsup Inequalities} \label{sec-liminf}
In this section we prove the following two lemmas, which establish the $\Gamma$-limit of $E_{\rm dGCH}^\vep$ as stated in Theorem \ref{th-main}.

\begin{lemma}\label{Liminf_Lemma}
{(Liminf Inequality)} For any function $u\in L^1(\Omega)$ and any sequence $\{u_k\}\subset L^1(\Omega)$  
such that $u_k \rightarrow u$ strongly in $L^1(\Omega)$, we have 
\begin{align}\label{liminf-ineq}
	E^0_{\rm dGCH}(u) \leq \liminf_{k\to\infty} E^{\vep_k}_{\rm dGCH}(u_k).
\end{align}
\end{lemma}

\begin{lemma}\label{Limsup_Lemma}
{(Limsup Inequality)} For any $u \in L^1(\Omega)$, 
there exists a sequence $\{u_k\} \subset L^1(\Omega)$ that converge to $u$ strongly in $L^1(\Omega)$ with 
\begin{align}\label{limsup-ineq}
 \limsup_{k\to\infty} E^{\vep_k}_{\rm dGCH}(u_k)\leq E^0_{\rm dGCH}(u).
\end{align} 
\end{lemma}

%
\subsection{Proof of the liminf inequality} 

If $\disp\liminf_{k\to\infty} E^{\vep_k}_{\rm dGCH}(u_k) = \infty$, then \qref{liminf-ineq} is trivial. So
we only need to consider the case  
$\liminf_{k\to\infty} E^{\vep_k}_{\rm dGCH}(u_k)<\infty$. 
By extracting 
a subsequence if necessary we may assume without loss of generality that
\[\lim_{k\to\infty} E^{\vep_k}_{\rm dGCH}(u_k) =  \liminf_{k\to\infty} E^{\vep_k}_{\rm dGCH}(u_k).\]
By the compactness result in Section \ref{sec-compactness}, there 
exists a subsequence $u_{k_j}$ and a set $A$ of finite perimeter such that 
$u_{k_j}$ converges to $u^-+(u^+-u^-)\chi_A$ strongly in $L^1(\Omega)$ and almost everywhere in 
$\Omega$. By the uniqueness of limit we conclude that actually $u=u^-+(u^+-u^-)\chi_A$, and 
\[E^0_{\rm dGCH} (u) = \sigma(p) Per(A).\]

{\it Case 1. $p=1$ and $g$ is in the singular form of \qref{def_g01}.  } 

By the lower 
semicontinuity of the BV semi-norm and \qref{bound3.1}, we get 
\begin{align*}
	E^0_{\rm dGCH}(u)&=\sqrt2 (u^+-u^-) Per(A) =\sqrt2 |u|_{BV(\Omega)} 
	\leq \liminf_{j\to\infty}\sqrt2 \int_\Omega |\nabla u_{k_j}|\,dx \nn\\
	&\leq \liminf_{j\to\infty} E^{\vep_{k_j}}_{\rm dGCH}(u_{k_j}) 
	=\liminf_{k\to\infty} E^{\vep_k}_{\rm dGCH}(u_k).\\
\end{align*}

{\it Case 2: $p=1$ and $g$ is in the regularized form of \qref{def-g_alpha}.}

By the AM-GM inequality, 
\begin{align} \label{est1}
	E^{\vep_k}_{\rm dGCH}(u_k)\geq \int_{\Omega} \sqrt{\frac{2[(u_{k}-u^-)(u_{k}-u^+)]^2}{[(u_{k}-u^-)(u_{k}-u^+)]^2
	+\alpha \vep_k^2}} |\nabla u_{k}| 
\end{align}
Similar to \qref{def-h} but without the cutoff, define 
\begin{align*}
	F_k(s)= \int_{u^-}^{s} \sqrt{\frac{2[(t-u^-)(t-u^+)]^2}{[(t-u^-)(t-u^+)]^2+\alpha \vep_k^2}} \,dt.
\end{align*}
Then $F_k$ is a Lipschitz function with a Lipschitz constant $\sqrt2$. 
Define 
\begin{align}
	w_k := F_k(u_k). \nn
\end{align}
Then $w_k \in W^{1,1}(\Omega)$ and 
\begin{align} \label{grad-w}
	|\na w_{k}|= \sqrt{\frac{2[(u_{k}-u^-)(u_{k}-u^+)]^2}{[(u_{k}-u^-)(u_{k}-u^+)]^2+\alpha \vep_k^2}} |\nabla u_{k}|
\end{align}
We will show that $w_k \to \sqrt2 (u-u^-) = \sqrt2 (u^+-u^-)\chi_A$ strongly in $L^1(\Omega)$. Indeed, 
\begin{align}
	&w_k - \sqrt 2 (u-u^-) = \int_{u^-}^{u_k} \sqrt{\frac{2[(t-u^-)(t-u^+)]^2}{[(t-u^-)(t-u^+)]^2+\alpha \vep_k^2}} \,dt 
	- \int_{u^-}^{u} \sqrt2 \,dt \nn\\
	&= \int_{u}^{u_k} \sqrt{\frac{2[(t-u^-)(t-u^+)]^2}{[(t-u^-)(t-u^+)]^2+\alpha \vep_k^2}} \,dt \nn\\
	&+ \int_{u^-}^{u} \left(\sqrt{\frac{2[(t-u^-)(t-u^+)]^2}{[(t-u^-)(t-u^+)]^2+\alpha \vep_k^2}} -\sqrt2\right)\,dt. \nn
\end{align}
Since $u=u^- $ in $\Omega\setminus A$ and $u=u^+$ in $A$, we have 
\begin{align}
	&\int_\Omega \left|w_k - \sqrt 2 (u-u^-)\right|\,dx \nn\\
	&\leq \sqrt2 \int_\Omega |u_k - u| \,dx
	+  |A| \int_{u^-}^{u^+} 
	\left|\sqrt{\frac{2[(t-u^-)(t-u^+)]^2}{[(t-u^-)(t-u^+)]^2+\alpha \vep_k^2}} -\sqrt2\right|\,dt \nn\\
	&\to 0 \qquad\mbox{as }k\to \infty. \nn
\end{align}
Here for the convergence of the second integral above, we use the Dominated Convergence Theorem. 
By the lower semicontinuity 
of the BV norm and \qref{est1}, \qref{grad-w}, we obtain the liminf inequality
\begin{align}
	\sqrt2 (u^+-u^-) Per(A) \leq \liminf_{k\to\infty} \int_\Omega |\nabla w_k|\,dx 
	\leq \liminf_{k\to\infty} E^{\vep_k}_{\rm dGCH}(u_k).
\end{align}

{\it Case 3. $0<p\leq 3/2$ and $g$ is in the singular form \qref{def_g0p}.}

Fix any $K>0$.  Let $h_K(t)$ be defined by \qref{def-h} and  
\[ w:=h_K(u), \quad \quad w_{k} := h_K(u_k)\quad \mbox{ for all }k,.\]
Since the Lipschitz constant of $h_K$ is $\sqrt2K$, 
\begin{align*}
	\int_\Omega |w_k -w| \,dx\leq \sqrt2K \int_\Omega |u_k - u|\,dx \to 0 \quad\mbox{as }k\to \infty.
\end{align*}
In addition, since $w=h_K(u^+)\chi_A$, by the lower semicontinuity of the BV norm, and \qref{bound12.1}, \qref{bound12.2}, 
we have 
\begin{align}\nn
	h_K(u^+) Per(A) &\leq \liminf_{k\to\infty} \int_\Omega |\nabla w_{k}|\,dx \nn\\
	&\leq \liminf_{j\to\infty} E^{\vep_{k_j}}_{\rm dGCH}(u_{k_j}) = \liminf_{k\to\infty} E^{\vep_k}_{\rm dGCH} (u_k). \nn
\end{align}
That is, 
\begin{align}\nn
	\sqrt2\left(\int_{u^-}^{u^+} \min\{|(s-u^-)(s-u^+)|^{1-p}, K\}\,ds\right) Per(A) \leq \liminf_{k\to\infty} E^{\vep_k}_{\rm dGCH} (u_k).
\end{align}
Letting $K\to\infty$, we obtain 
\begin{align}\nn
	\sqrt2 \left(\int_{u^-}^{u^+} |(s-u^-)(s-u^+)|^{1-p}\,ds\right) Per(A) \leq \liminf_{k\to\infty} E^{\vep_k}_{\rm dGCH} (u_k).
\end{align}
This is the liminf inequality \qref{liminf-ineq} for $0<p\leq 3/2.$

\subsection{The recovery sequence and the limsup inequality} \label{sec:recovery}

For any set $A\subset\Omega$ of finite perimeter, we have a corresponding $u:= u^-+(u^+-u^-)\chi_A \in BV(\Omega)$. 
We need to find a corresponding recovery sequence $u_k$ in the sense that $u_k\to u$ strongly in $L^1(\Omega)$ and 
\begin{align}\label{recovery}
	E^0_{\rm dGCH}(u) \ge \limsup_{k\to\infty} E^{\vep_k}_{\rm dGCH}(u_k). 
\end{align}
Note that if we can establish such sequences for $p=1$ and $0<p<1$, then the regularized dGCH will work under the same recovery sequence.  We construct the sequence for $0<p<1$.  

Suppose $u \in L^1(\Omega)$ with $E_{\rm dGCH}^0(u)$ finite.  Then $u=u^-+(u^+-u^-)\chi_A$ for some set $A$ with finite perimeter. 
We wish to construct some $u_{k}$ that converges to $u$ in the $L^1(\Omega)$ sense and the limsup condition holds.  There are multiple ways 
to construct sequences of recovery functions, see the dissertation of the second author \cite{renzi:thesis}. Here we will show that the 
standard construction for Cahn-Hilliard energy functionals also works for the dGCH functional.

Define $\Phi_k:[u^-,u^+] \rightarrow [0, \Phi_k(u^+)]$ as follows:
\begin{align*}
	\Phi_k(t)=\int_{u^-}^t\frac{\vep_k}{\sqrt{2W(s)+\vep_k}}\;ds.
\end{align*}
$\Phi_k$ is strictly increasing and continuous and therefore invertible.  Moreover, $0<\Phi_k(u^+)<\sqrt{\vep_k}(u^+-u^-)<\infty$.  Also, $\Phi_k(u^+) \rightarrow 0$ as $k\rightarrow \infty$.  So, we can define $\phi_k=\Phi_k^{-1}.$  We extend $\phi_k$ to the real line by setting $\phi_k(t)=u^-$ for $t \leq 0$ and $\phi_k(t)=u^+$ for $t>\Phi_k(u^+)$.
We can find the derivative by using the inverse rule:
\begin{align*}
	\frac{d\phi_k}{dt}&=\frac{\sqrt{2(W(\phi_k(t))+\vep_k)}}{\vep_k}.
\end{align*}
Now, let $d_A(x)$ be the following function
\begin{align*}
	d_A(x)=\left\{ \begin{array}{cc} 
		\text{dist}(x,\partial A) & \hspace{5mm} x \in A \\
		-\text{dist}(x,\partial A) & \hspace{5mm} x \in \Omega\setminus  A
	\end{array} \right.
\end{align*}
and define
\begin{align*}
	u_{k}=\phi_{k}(d_A(x)).
\end{align*}
Then $u_k \rightarrow \chi_A$ both pointwise in $\Omega$ and strongly in $L^1(\Omega)$.  
If $\partial A$ is $C^2$, then
\begin{align*}
	|\na u_k|=|\phi_k'(d_A(x) \na d_A(x)|&=\frac{\sqrt{2(W(\phi_k(d_A(x)))+\vep_k)}}{\vep_k}|\na d_A(x)|\\
	&=\frac{\sqrt{2(W(\phi_k(d_A(x)))+\vep_k)}}{\vep_k}
\end{align*}
Furthermore,
\begin{align*}
	&\displaystyle\lim_{k\rightarrow\infty} \sup_{u^- < s < u^+} \mathcal{H}^{n-1}(\{x\in\Omega:u_k(x)=s\}) \\
	&= \lim_{k\rightarrow\infty} \sup_{0 < s < \Phi_k(u^+)} \mathcal{H}^{n-1}(\{x\in\Omega:d_A(x)=\Phi_k(s)\}) \\
	&=P_{\Omega}(A)
\end{align*}
So, using the coarea formula, we get:
\begin{align*}
	&\lim_{k\rightarrow\infty} E^{\vep_k}_{\text{dGCH}}(u_k) \\
	&= \displaystyle\int_{\Omega} \left(\frac{\vep_k | \na u_{k} |^2}{2|u_{k}-u^-|^{p}|u_{k}-u^+|^{p}}+\frac{|u_{k}-u^-|^{2-p}|u_{k}-u^+|^{2-p}}{\vep_k}\right)dx \\
	&=\lim_{k\rightarrow\infty} \displaystyle\int_{\Omega} \left(\frac{\vep_k | \na u_{k} |}{2}+\frac{|u_{k}-u^-|^{2}|u_{k}-u^+|^{2}}{\vep_k|\na u_k|}\right)\frac{|\na u_k|}{|u_{k}-u^-|^{p}|u_{k}-u^+|^{p}}\;dx 	\\
	&=\lim_{k\rightarrow\infty} \displaystyle\int_{\Omega} \left(\frac{\sqrt{W(u_k)+\vep_k}}{\sqrt{2}}+\frac{W(u_k)}{\sqrt{2(W(u_k)+\vep_k)}}\right)\frac{|\na u_k|}{|u_{k}-u^-|^{p}|u_{k}-u^+|^{p}} dx	\\
	&=\lim_{k\rightarrow\infty} \displaystyle\int_{\Omega} \left(\frac{2W(u_k)+\vep_k}{\sqrt{2(W(u_k)+\vep_k)}}\right)\frac{|\na u_k|}{|u_{k}-u^-|^{p}|u_{k}-u^+|^{p}} dx	\\
	&=\lim_{k\rightarrow\infty} \displaystyle\int_{\Omega} \Big(\frac{\sqrt{2W(u_k)}}{|u_{k}-u^-|^{p}|u_{k}-u^+|^{p}}\Big)|\na u_k| \,dx \\
	&=\lim_{k\rightarrow\infty} \displaystyle\int_{-\infty}^\infty \Big(\frac{\sqrt{2W(r)}}{|r-u^-|^{p}|r-u^+|^{p}}\Big) \mathcal{H}^{n-1}(\{x\in \Omega : u_k(x)=r \})dr\\
	&=\lim_{k\rightarrow\infty} \displaystyle\int_{u^-}^{u^+} \sqrt{2} |r-u^-|^{1-p}|r-u^+|^{1-p} \mathcal{H}^{n-1}(\{x\in \Omega : u_k(x)=r \})dr\\
	&\leq \sqrt{2} P_{\Omega}(A) \displaystyle\int_{u^-}^{u^+} |r-u^-|^{1-p}|r-u^+|^{1-p} dr.
\end{align*}
This is exactly what we want.

We still need to consider the case where $A$ doesn't have $C^2$ boundary.  Here we cite Lemma 1.15 from 
\cite{leoni:note}.  We can find a sequence of open $A_j$ with $\partial A_j$ a nonempty, compact hypersurface of class 
$C^2$ and $\mathcal{H}^{n-1}(\partial A_j \cap \partial \Omega)=0$ such that $\chi_{A_j} \rightarrow \chi_A$ in 
$L^1(\Omega)$, $\text{Per}(A_j,\Omega) \rightarrow \text{Per}(A, \Omega),$ and $L^n(A_j)=L^n(A)$ for all $j$.
We can find a sequence $\{u_{k,j}\} \subset H^1(\Omega)$ such that $u_{k,j}\rightarrow \chi_{B_j}$ with the limsup 
condition satisfied.  Now, 
by a diagonalization argument, we can find a sequence of $u_{k,j}$ that still satisfies the limsup condition for $u$ 
and converge to $u$.
A similar argument works for $p=1$.  And if the argument holds for $p=1$, the same sequence will work for the 
regularized version of the dGCH equation.

\section{Discussion} In this paper we establish the sharp-interface limit of the dGCH energy in 
the framework of $\Gamma$--convergence. This is the first step to understand the limiting behavior as $\vep\to 0$ of the dGCH energy. 
There are many open questions. One is the properties of the minimizers of the dGCH energy under various boundary conditions. 
For instance, under the strong anchoring conditions, it is not surprising to expect the minimizers to have similar properties as those 
shown in \cite{Dai-Li-Luong:minimizers} for the ordinary Cahn-Hilliard energy. However, due to the de Gennes coefficient, the analysis would be much 
more complicated. Our ultimate goal is to analyze the limiting behavior as $\vep\to 0$ of the degenerate de Gennes-Cahn-Hilliard equation. One 
step before that would be to study the force convergence as is done in \cite{dll:conv} for a phase field model for molecule solvations. It will also be of 
interest to study a corresponding $L^2$ gradient flow and its degenerate version, which can be called the de Gennes-Allen-Cahn equations. We will leave these topics for future studies.

\bibliographystyle{plain}


\end{document}